\documentclass{amsart}
\usepackage{mathrsfs}
\usepackage{amssymb} %Extra symbols 
\usepackage{url}
\usepackage{hyperref}
\usepackage{color}
\definecolor{refcolor}{RGB}{0,0,190}
\hypersetup{
    colorlinks,
    citecolor=refcolor,
    filecolor=refcolor,
    linkcolor=refcolor,
    urlcolor=refcolor
}

\theoremstyle{definition}
\newtheorem{theorem}{Theorem}[section]
\newtheorem{definition}[theorem]{Definition}
\newtheorem{proposition}[theorem]{Proposition}

\newtheorem{corollary}[theorem]{Corollary}
\newtheorem{remark}[theorem]{Remark}
\newtheorem{example}[theorem]{Example}

\def\({\left(}
\def\){\right)}

\newcommand{\R}{\mathbb{R}}

\newcommand{\de}{\textnormal{d}}

\newcommand{\grad}{\textnormal{grad }}

\newcommand{\tn}{\textnormal}
\newcommand{\ds}{\displaystyle}

\newcommand{\ie}{\textit{i.e.} }
\newcommand{\cf}{\textit{cf.} }
\newcommand{\eg}{\textit{e.g.} }

\newcommand{\citep}[2]{\cite{#1}, p. #2}
\newcommand{\citepcf}[2]{(\cf \citep{#1}{#2})}

\newcommand{\rank}{\textnormal{rank }}

\newcommand{\mf}[1]{\mathfrak{#1}}
\newcommand{\mc}[1]{\mathcal{#1}}
\newcommand{\ms}[1]{\mathscr{#1}}

\newcommand{\tensors}[3]{\mc T{}^{#1}_{#2}#3}

\newcommand{\sref}[1]{\S\ref{#1}}

\newcommand{\idxannih}[2]{#1{}^{#2}{}}
\newcommand{\idxcoannih}[2]{#1{}_{#2}{}}

\newcommand{\radix}[1]{\idxcoannih{#1}{\circ}}
\newcommand{\annih}[1]{\idxannih{#1}{\bullet}}
\newcommand{\coannih}[1]{\idxcoannih{#1}{\bullet}}

\newcommand{\annihg}{\coannih{g}}

\newcommand{\metric}[1]{\langle#1\rangle}
\newcommand{\annihprod}[1]{\coannih{\langle\!\langle#1\rangle\!\rangle}}

\newcommand{\supp}[1]{\tn{supp}(#1)}

\newcommand{\cocontr}{{{}_\bullet}}

\newcommand{\vectmodule}{\mf X}
\newcommand{\fivect}[1]{\vectmodule(#1)}
\newcommand{\fivectnull}[1]{\vectmodule_\circ(#1)}
\newcommand{\fiscal}[1]{\ms F(#1)}

\newcommand{\fiformk}[2]{\mc A^{#1}(#2)}

\newcommand{\fivectlift}[1]{\mf L(#1)}
\newcommand{\annihforms}[1]{\annih{\mc A}(#1)}

\newcommand{\discformsk}[2]{A_d{}^{#1}(#2)}

\newcommand{\metricformsk}[2]{\annih{\ms A}{}^{#1}(#2)}
\newcommand{\srformsk}[2]{\annih{\ms A}{}^{#1}(#2)}

\newcommand{\lie}{\mc L}
\newcommand{\kosz}{\mc K}
\newcommand{\der}{\nabla}
\newcommand{\dera}[1]{\der_{#1}}

\newcommand{\lder}{\der^{\flat}}
\newcommand{\ldera}[1]{\lder_{#1}}
\newcommand{\lderb}[2]{\ldera{#1}{#2}}
\newcommand{\lderc}[3]{(\lderb{#1}{#2})(#3)}

\newcommand{\ric}{\tn{Ric}}

\newcommand{\dsfrac}[2]{\ds{\frac{#1}{#2}}}

\newcommand{\schw}{Schwarzschild}
\newcommand{\rn}{Reissner-Nordstr\"om}
\newcommand{\kn}{Kerr-Newman}
\newcommand{\flrw}{Friedmann-Lema\^itre-Robertson-Walker}

\def\hyph{-\penalty0\hskip0pt\relax}
\newcommand{\semiriem}{semi{\hyph}Riemannian}

\newcommand{\semireg}{semi{\hyph}regular}
\newcommand{\ssemireg}{Semi{\hyph}regular}
\newcommand{\nondeg}{non{\hyph}degenerate}

\newcommand{\rstationary}{radical{\hyph}stationary}

\newcommand{\rannih}{radical{\hyph}annihilator}

\begin{document} 
 
%--------------------------------------------------------
% Title
\title{The Geometry of Warped Product Singularities}
\author{O.C. Stoica$^1$}
\thanks{$^1$Department of Theoretical Physics, National Institute of Physics and Nuclear Engineering -- \textit{Horia Hulubei}, Bucharest, Romania, e-mail: cristi.stoica@theory.nipne.ro.\\Partially supported by Romanian Government grant PN II Idei 1187.}

\begin{abstract}
In this article the degenerate warped products of singular semi-Riemannian manifolds are studied. They were used recently by the author to handle singularities occurring in General Relativity, in black holes and at the big-bang. One main result presented here is that a degenerate warped product of semi-regular semi-Riemannian manifolds with the warping function satisfying a certain condition is a semi-regular semi-Riemannian manifold. The connection and the Riemann curvature of the warped product are expressed in terms of those of the factor manifolds. Examples of singular semi-Riemannian manifolds which are semi-regular are constructed as warped products. Applications include cosmological models and black holes solutions with semi-regular singularities. Such singularities are compatible with a certain reformulation of the Einstein equation, which in addition holds at semi-regular singularities too.

\bigskip
\noindent 
\end{abstract}

%--------------------------------------------------------
% Title and contents

\maketitle

\setcounter{tocdepth}{1}
\tableofcontents

%~~~~~~~~~~~~~~~~~~~~~~~~~~~~~~~~~~~~~~~~~~~~~~~~~~~~~~~~~~~~~~~~~~~~~~~%
\section{Introduction}

The warped product provides a way to construct new semi-Riemannian manifolds from known ones \cite{BON69,BEP82,ONe83}. This construction has useful applications in General Relativity, in the study of cosmological models and black holes. In such models, singularities are usually present, and at such points the warping function becomes $0$. For the {\flrw} model for example, the metric of the product manifold becomes degenerate, and the Levi-Civita connection and Riemann curvature, as usually defined, become singular or undefined. Therefore, we need to apply the tools of singular geometry \cite{Sto11a}.

This article continues the study of singular manifolds developed by the author in \cite{Sto11a,Sto11d}, extending it to warped products. We start with a brief recall of notions related to product manifolds in \sref{ss_prod_man}. Then, basic notions of singular geometry and the main ideas from \cite{Sto11a}, which will be applied here, are remembered in \sref{ss_singular_semi_riem}. In \sref{s_deg_wp_ssr} we define the degenerate warped products of singular manifolds, and study the Koszul form of the warped product in terms of the Koszul form of the factors. Then, in \sref{s_semi_reg_semi_riem_man_warped} we show that the warped products of {\rstationary} manifolds are also {\rstationary}, if the warping function satisfies a certain condition. After that, we prove a similar result for {\semireg} manifolds, which ensures the smoothness of the Riemann curvature tensor. In \sref{s_riemann_wp_deg} we express the Riemann curvature of {\semireg} warped products in terms of the factor manifolds.
We conclude in \sref{s_riemann_wp_deg_applications} by giving some examples of {\semireg} warped products, and some applications to General Relativity, including cosmological models with {\semireg} big-bang singularity, and stationary black hole solutions. {\ssemireg} singularities are compatible with a densitized version of Einstein's equation, which remains valid at the singularity too.

%~~~~~~~~~~~~~~~~~~~~~~~~~~~~~~~~~~~~~~~~~~~~~~~~~~~~~~~~~~~~~~~~~~~~~~~%
\section{Preliminaries}
\label{s_prelim}

%~~~~~~~~~~~~~~~~~~~~~~~~~~~~~~~~~~~~~~~~~~~~~~~~~~~~~~~~~~~~~~~~~~~~~~~%
\subsection{Product manifolds}
\label{ss_prod_man}

We first recall some elementary notions about the \textit{product manifold} $B\times F$ of two differentiable manifolds $B$ and $F$. See for example \cite{ONe83}, p. {24--25}.

At each point $p=(p_1,p_2)$ of the manifold $M_1\times M_2$, the tangent space decomposes as
\begin{equation}
	T_{(p_1,p_2)}(M_1\times M_2)\cong T_{(p_1,p_2)}(M_1)\oplus T_{(p_1,p_2)}(M_2),
\end{equation}
where $T_{(p_1,p_2)}(M_1):=T_{(p_1,p_2)}(M_1\times p_2)$ and $T_{(p_1,p_2)}(M_2):=T_{(p_1,p_2)}(p_1\times M_2)$.

Let $\pi_i:M_1\times M_2\to M_i$, for $i\in\{1,2\}$, be the canonical projections. The \textit{lift of the scalar field} $f_i\in\fiscal{M_i}$  is the scalar field $\tilde f_i:=f_i\circ\pi_i\in\fivect{M_1\times M_2}$.
The \textit{lift of the vector field} $X_i\in\fivect{M_i}$ is the unique vector field $\tilde X_i$ on $M_1\times M_2$ satisfying $\de \pi_i(\tilde X_i)=X_i$. We denote the set of all vector fields $X\in \fivect{M_1\times M_2}$ which are lifts of vector fields $X_i\in \fivect{M_i}$ by $\fivectlift{M,M_i}$. The \textit{lift of a covariant tensor} $T\in\tensors 0 s {M_i}$ is given by $\tilde T\in \tensors 0 s {(M_1 \times M_2)}$, $\tilde T:=\pi^*_i(T)$. The \textit{lift of a tensor} $T\in\tensors 1 s {M_i}$ is given, for any $X_1,\ldots,X_s\in \fivect{M_1 \times M_2}$, by $\tilde T\in \tensors 1 s {(M_1 \times M_2)}$, $\tilde T(X_1,\ldots,X_s) =\tilde X$, where $\tilde X \in \fivect{M_1 \times M_2}$ is the lifting of the vector field $X\in\fivect{M_i}$, $X=T(\pi_i(X_1),\ldots,\pi_i(X_s))$.

%~~~~~~~~~~~~~~~~~~~~~~~~~~~~~~~~~~~~~~~~~~~~~~~~~~~~~~~~~~~~~~~~~~~~~~~%
\subsection{Singular {\semiriem} manifolds}
\label{ss_singular_semi_riem}

We recall here some notions about singular semi-Riemannian manifolds, and some of the main results from \cite{Sto11a}, which will be used in the rest of the article.

\begin{definition}(also see \cite{Kup87b})
\label{def_sing_semiRiemm_man}
A \textit{singular {\semiriem} manifold} $(M,g)$ is a differentiable manifold $M$ endowed with a symmetric bilinear form $g\in \Gamma(T^*M \odot_M T^*M)$ named \textit{metric}. The manifold $(M,g)$ is said to be with \textit{constant signature} if the signature of $g$ is fixed, otherwise, $(M,g)$ is said to be with \textit{variable signature}. Particular cases are the \textit{{\semiriem} manifolds}, having the metric {\nondeg} (and automatically having constant signature), and  \textit{Riemannian manifolds}, when $g$ is positive definite.
\end{definition}

If $(V,g)$ is a finite dimensional inner product space with an inner product $g$ which may be degenerate, then we call the totally degenerate space $\radix{V}:=V^\perp$ the \textit{radical} of $V$. The inner product $g$ on $V$ is {\nondeg} if and only if $\radix{V}=\{0\}$.
The \textit{radical of $TM$}, denoted by $\radix{T}M$, is defined by $\radix{T}M=\cup_{p\in M}\radix{(T_pM)}$. We denote by $\fivectnull{M}$ the module of vector fields on $M$ for which $W_p\in\radix{(T_pM)}$. 

The remaining of this section recalls very briefly the main notions and results on singular manifolds, as presented in \cite{Sto11a}.

We define
\begin{equation}
	\annih{T}M=\bigcup_{p\in M}\annih{(T_pM)},
\end{equation}
where $\annih{(T_pM)} \subseteq T^*_pM$ is the space of covectors at $p$ of the form $\omega_p(X_p)=\metric{Y_p,X_p}$, for some vectors $Y_p\in T_p M$ and any $X_p\in T_p M$. We define sections of $\annih{T}M$ by
\begin{equation}
	\annihforms{M}:=\{\omega\in\fiformk 1{M}|\omega_p\in\annih{(T_pM)}\tn{ for any }p\in M\}.
\end{equation}
On $\annih{T}M$ there is a unique {\nondeg} inner product $\annihg$, defined by $\annihprod{\omega,\tau}:=\annihg(\omega,\tau):=\metric{X,Y}$, where $\annih X=\omega$, $\annih Y=\tau$, $X,Y\in\fivect M$.

A tensor $T$ of type $(r,s)$ is named \textit{{\rannih}} in the $l$-th covariant slot if  $T\in \tensors r{l-1}{M}\otimes_M\annih{T}M\otimes_M \tensors 0{s-l}{M}$.

We now show how to define uniquely the \textit{covariant contraction} or \textit{covariant trace}. We define it first on tensors $T\in\annih{T}M\otimes_M\annih{T}M$, by $C_{12}T=\annihg^{ab}T_{ab}$. This definition does not depend on the basis, because $\annihg\in\annih{T}^*M\otimes_M\annih{T}^*M$. This operation can be extended by linearity to any tensors which are radical in two covariant indices. For a tensor field $T$ we define the contraction $C_{kl} T$ by
\begin{equation*}
T(\omega_1,\ldots,\omega_r,v_1,\ldots,\cocontr,\ldots,\cocontr,\ldots,v_s).
\end{equation*}

If the metric is non-degenerate, we can define the covariant derivative of a vector field $Y$ in the direction of a vector field $X$, where $X,Y\in\fivect{M}$, by the \textit{Koszul formula} (see \eg \citep{ONe83}{61}). If the metric is degenerate, we cannot extract the covariant derivative from the Koszul formula. We define the \textit{Koszul form} as a shorthand for the long right part of the Koszul formula and were emphasized some of its properties.

Let's recall the definition of the Koszul form and its properties, without proof, from \cite{Sto11a}.

\begin{definition}[The Koszul form]
\label{def_Koszul_form}
\textit{The Koszul form} is defined as
\begin{equation*}
	\kosz:\fivect M^3\to\R,
\end{equation*}
\begin{equation}
\label{eq_Koszul_form}
\begin{array}{llll}
	\kosz(X,Y,Z) &:=&\ds{\frac 1 2} \{ X \metric{Y,Z} + Y \metric{Z,X} - Z \metric{X,Y} \\
	&&\ - \metric{X,[Y,Z]} + \metric{Y, [Z,X]} + \metric{Z, [X,Y]}\}.
\end{array}
\end{equation}
\end{definition}

\begin{theorem}
\label{thm_Koszul_form_props}
Properties of the Koszul form of a singular {\semiriem} manifold $(M,g)$:
\begin{enumerate}
	\item \label{thm_Koszul_form_props_linear}
	Additivity and $\R$-linearity in each of its arguments.
	\item \label{thm_Koszul_form_props_flinearX}
	$\fiscal M$-linearity in the first argument:

	$\kosz(fX,Y,Z) = f\kosz(X,Y,Z).$
	\item \label{thm_Koszul_form_props_flinearY}
	The \textit{Leibniz rule}:

	$\kosz(X,fY,Z) = f\kosz(X,Y,Z) + X(f) \metric{Y,Z}.$
	\item \label{thm_Koszul_form_props_flinearZ}
	$\fiscal M$-linearity in the third argument:

	$\kosz(X,Y,fZ) = f\kosz(X,Y,Z).$
	\item \label{thm_Koszul_form_props_commutYZ}
	It is \textit{metric}:

	$\kosz(X,Y,Z) + \kosz(X,Z,Y) = X \metric{Y,Z}$.
	\item \label{thm_Koszul_form_props_commutXY}
	It is \textit{symmetric}:

	$\kosz(X,Y,Z) - \kosz(Y,X,Z) = \metric{[X,Y],Z}$.
	\item \label{thm_Koszul_form_props_commutZX}
	Relation with the Lie derivative of $g$:

	$\kosz(X,Y,Z) + \kosz(Z,Y,X) = (\lie_Y g)(Z,X)$.
	\item \label{thm_Koszul_form_props_commutX2Y}

	$\kosz(X,Y,Z) + \kosz(Y,Z,X) = Y\metric{Z,X} + \metric{[X,Y],Z}$.
	\end{enumerate}
for any $X,Y,Z\in\fivect M$ and $f\in\fiscal M$.
\qed\end{theorem}

\begin{definition}
\label{def_l_cov_der}
Let $X,Y\in\fivect M$. The \textit{lower covariant derivative} of $Y$ in the direction of $X$ is defined as the differential $1$-form $\lderb XY \in \fiformk 1{M}$
\begin{equation}
\label{eq_l_cov_der_vect}
\lderc XYZ := \kosz(X,Y,Z),
\end{equation}
for any $Z\in\fivect{M}$.
We also define the \textit{lower covariant derivative operator}
\begin{equation}
	\lder:\fivect{M} \times \fivect{M} \to \fiformk 1{M},
\end{equation}
which associates to each $X,Y\in\fivect{M}$ the differential $1$-form $\ldera XY$.
\end{definition}

\begin{definition}
\label{def_radical_stationary_manifold}
A singular manifold $(M,g)$ is \textit{{\rstationary}} if it satisfies the condition 
\begin{equation}
\label{eq_radical_stationary_manifold}
		\kosz(X,Y,\_)\in\annihforms M,
\end{equation}
for any $X,Y\in\fivect{M}$. This definition is more general than Definition 3.1.3 from \cite{Kup96}, because it is not limited to constant signature metrics.
\end{definition}

\begin{definition}
\label{def_cov_der_covect}
Let $X\in\fivect{M}$, $\omega\in\annihforms{M}$, where $(M,g)$ is {\rstationary}. The covariant derivative of $\omega$ in the direction of $X$ is defined as
\begin{equation*}
	\der:\fivect{M} \times \annihforms{M} \to \discformsk 1 M,
\end{equation*}
\begin{equation}
	\left(\der_X\omega\right)(Y) := X\left(\omega(Y)\right) - \annihprod{\lderb X Y,\omega},
\end{equation}
where $\discformsk 1 M$ denotes the set of $1$-forms which are smooth on the regions of constant signature.
\end{definition}

\begin{definition}
\label{def_cov_der_smooth}
If the singular {\semiriem} manifold $(M,g)$ is {\rstationary}, we define:
\begin{equation}
	\srformsk 1 M = \{\omega\in\annihforms M|(\forall X\in\fivect M)\ \der_X\omega\in\annihforms M\},
\end{equation}
\begin{equation}
	\srformsk k M := \bigwedge^k_M\srformsk 1 M.
\end{equation}
\end{definition}

\begin{definition}
\label{def_riemann_curvature}
The \textit{Riemann curvature tensor} is defined as
\begin{equation*}
	R: \fivect M\times \fivect M\times \fivect M\times \fivect M \to \R,
\end{equation*}
\begin{equation}
\label{eq_riemann_curvature}
	R(X,Y,Z,T) := (\dera X {\ldera Y}Z - \dera Y {\ldera X}Z - \ldera {[X,Y]}Z)(T)
\end{equation}
for any vector fields $X,Y,Z,T\in\fivect{M}$.
\end{definition}

\begin{definition}
\label{def_semi_regular_semi_riemannian}
A singular {\semiriem} manifold $(M,g)$ satisfying
\begin{equation}
	\ldera X Y \in\srformsk 1 M
\end{equation}
for any vector fields $X,Y\in\fivect{M}$ is called \textit{{\semireg} {\semiriem} manifold}.
\end{definition}

\begin{proposition}
\label{thm_sr_cocontr_kosz}
A {\rstationary} {\semiriem} manifold $(M,g)$ is {\semireg} if and only if for any $X,Y,Z,T\in\fivect M$
\begin{equation}
	\kosz(X,Y,\cocontr)\kosz(Z,T,\cocontr) \in \fiscal M.
\end{equation}
\qed
\end{proposition}

\begin{example}
\label{s_semi_reg_semi_riem_man_example_diagonal}
We construct a useful example of {\semireg} metric \cite{Sto11a}. Let's consider that there is a coordinate chart in which the metric is diagonal. The components of the Koszul form are in this case the Christoffel's symbols of the first kind, which are of the form $\pm\frac 1 2\partial_a g_{bb}$, because the metric is diagonal. Assume that $g=\sum_a\varepsilon_a\alpha_a^2\de x^a\otimes \de x^a$, $\varepsilon_a\in\{-1,1\}$. Then the metric is {\semireg} if there is a smooth function $f_{abc}\in\fiscal{M}$ with $\supp{f_{abc}}\subseteq\supp{\alpha_c}$ for any $a,b\in\{1,\ldots,n\}$ and $c\in\{a,b\}$, and
\begin{equation}
\label{eq_diagonal_metric:semireg}
\partial_a\alpha_b^2=f_{abc}\alpha_c.
\end{equation}
If $c=b$, $\partial_a\alpha_b^2=2\alpha_b\partial_a\alpha_b$ implies that the function is  $f_{abb}=2\partial_a\alpha_b$. In addition, this has to satisfy the condition $\partial_a\alpha_b=0$ whenever $\alpha_b=0$. We require the condition $\supp{f_{abc}}\subseteq\supp{\alpha_c}$ because for being {\semireg}, a manifold has to be {\rstationary}.
\end{example}

\begin{theorem}
\label{thm_riemann_curvature_semi_regular}
The Riemann curvature of a {\semireg} {\semiriem} manifold $(M,g)$ is a smooth tensor field $R\in\tensors 0 4 M$.
\qed
\end{theorem}

\begin{proposition}
\label{thm_riemann_curvature_tensor_koszul_formula}
The Riemann curvature of a {\semireg} {\semiriem} manifold $(M,g)$ satisfies 
\begin{equation}
\begin{array}{lll}
	R(X,Y,Z,T) &=& X\left(\lderc Y Z T\right) - Y\left(\lderc X Z T\right) - \lderc {[X,Y]}ZT \\
&& + \annihprod{\lderb XZ,\lderb Y T} - \annihprod{\lderb YZ,\lderb X T}, \\
\end{array}
\end{equation}
and
\begin{equation}
\label{eq_riemann_curvature_tensor_koszul_formula}
\begin{array}{lll}
	R(X,Y,Z,T)&=& X \kosz(Y,Z,T) - Y \kosz(X,Z,T) - \kosz([X,Y],Z,T)\\
	&& + \kosz(X,Z,\cocontr)\kosz(Y,T,\cocontr) - \kosz(Y,Z,\cocontr)\kosz(X,T,\cocontr),
\end{array}
\end{equation}
for any vector fields $X,Y,Z,T\in\fivect{M}$.
\qed
\end{proposition}

%~~~~~~~~~~~~~~~~~~~~~~~~~~~~~~~~~~~~~~~~~~~~~~~~~~~~~~~~~~~~~~~~~~~~~~~%
\section{Degenerate warped products of singular {\semiriem} manifolds}
\label{s_deg_wp_ssr}

The warped product is defined in general between two ({\nondeg}) {\semiriem} manifolds, (\cf \cite{BON69}, \cite{BEP82}, \cite{kruchkovich1961spacesKagan}, \cite{kruchkovich1961classRiemannSpaces}, \cite{ONe83}, p. {204--211}, also see the survey in \cite{pranovic1995warped} and references therein. It is straightforward to extend the definition to singular {\semiriem} manifolds, as it is done in this section.
Further, we will study some properties of the warped products, in situations when the warping function $f$ is allowed to vanish or to become negative, and when $(B,g_B)$ and $(F,g_F)$ are allowed to be singular and with variable signature.

\begin{definition}[generalizing \cite{ONe83}, p. {204}]
\label{def_wp}
Let $(B,g_B)$ and $(F,g_F)$ be two singular {\semiriem} manifolds, and $f\in\fiscal{B}$ a smooth function. The \textit{warped product} of $B$ and $F$ with \textit{warping function} $f$ is the {\semiriem} manifold
\begin{equation}
	B\times_f F:=\big(B\times F, \pi^*_B(g_B) + (f\circ \pi_B)\pi^*_F(g_F)\big),
\end{equation}
where $\pi_B:B\times F \to B$ and $\pi_F: B \times F \to F$ are the canonical projections. It is customary to call $B$ the \textit{base} and $F$ the \textit{fiber} of the warped product $B\times_f F$.

We will use for all vector fields $X_B,Y_B\in\fivect{B}$ and $X_F,Y_F\in\fivect{F}$ the notation $\metric{X_B,Y_B}_B := g_B(X_B,Y_B)$ and $\metric{X_F,Y_F}_F := g_F(X_F,Y_F)$.
The inner product on $B\times_f F$ takes, for any point $p\in B\times F$ and for any pair of tangent vectors $x,y\in T_p(B\times F)$, the explicit form
\begin{equation}
\label{eq_wp_metric}
	\metric{x,y}=\metric{\de \pi_B(x),\de \pi_B(y)}_B + f^2(p)\metric{\de \pi_F(x),\de \pi_F(y)}_F.
\end{equation}
\end{definition}

\begin{remark}
\label{rem_wp}
The metric of the degenerate warped product from Definition \ref{def_wp} has the form
\begin{equation}
\de s_{B\times F}^2 = \de s_B^2 + f^2\de s_F^2.
\end{equation}
\end{remark}

\begin{remark}
Definition \ref{def_wp} is a generalization of the warped product definition, which is usually given for the case when both $g_B$ and $g_F$ are {\nondeg} and $f>0$ (see \cite{BON69}, \cite{BEP82} and \cite{ONe83}). In our definition these restrictions are dropped.
\end{remark}

\begin{remark}[similar to \cite{ONe83}, p. {204--205}]
For any $p_B\in B$, $\pi_B^{-1}(p_B)=p_B\times F$ is named the \textit{fiber} through $p_B$ and it is a {\semiriem} manifold. $\pi_F|_{p_B\times F}$ is a (possibly degenerate) homothety onto $F$. For each $p_F\in F$, $\pi_F^{-1}(p_F)=B\times p_F$ is a {\semiriem} manifold named the \textit{leave} through $p_F$. $\pi_B|_{B\times p_F}$ is an isometry onto $B$. For each $(p_B,p_F)\in B\times F$, $B \times p_F$ and $p_B \times F$ are orthogonal at $(p_B,p_F)$. For simplicity, if a vector field is a lift, we will use sometimes the same notation if they can be distinguished from the context. For example, we will be using $\metric{V,W}_F:=\metric{\pi_F(V),\pi_F(W)}_F$ for $V,W\in\fivectlift{B \times F,F}$.
\end{remark}

The following proposition recalls some simple facts which will be used frequently in the following.

\begin{proposition}
\label{thm_wp_deg_fundam}
Let $B \times_f F$ be a degenerate warped product, and let be the vector fields $X,Y,Z\in\fivectlift{B \times F,B}$ and $U,V,W\in\fivectlift{B \times F,F}$. Then
\begin{enumerate}
	\item \label{thm_wp_deg_fundam:mixed_metric}
	$\metric{X,V}=0$.
	\item \label{thm_wp_deg_fundam:mixed_lie_bracket}
	$[X,V]=0$.
	\item \label{thm_wp_deg_fundam:metric_B_constant_F}
	$V\metric{X,Y}=0$.
	\item \label{thm_wp_deg_fundam:metric_F_var_B}
	$X\metric{V,W}=2f\metric{V,W}_F X(f)$.
\end{enumerate}
\end{proposition}
\begin{proof}
	\eqref{thm_wp_deg_fundam:mixed_metric} and \eqref{thm_wp_deg_fundam:mixed_lie_bracket} are evident because the manifold is $B\times F$.
	
	\eqref{thm_wp_deg_fundam:metric_B_constant_F}
	$\metric{X,Y}=\metric{X,Y}_B$ is constant on fibers, and $V\metric{X,Y}=0$ because $V$ is vertical.
	
	\eqref{thm_wp_deg_fundam:metric_F_var_B}
	$X\metric{V,W}=X(f^2\metric{V,W}_F)=2f\metric{V,W}_F X(f)$.
\end{proof}

The properties in the following propositions are similar to some properties of the Levi-Civita connection for the warped product of ({\nondeg}) {\semiriem} manifolds {\cf} \eg \cite{ONe83}, p. {206}, but in addition are valid for the degenerate case too. These properties and their proofs in the regular case can't be adapted immediately, because for degenerate metric the Levi-Civita connection is not defined, and also we need to avoid the index raising.
But by rewriting them in terms of the Koszul form, rather than the Levi-Civita connection, they work for degenerate warped products too.

\begin{proposition}
\label{thm_wp_deg_koszul}
Let $B \times_f F$ be a degenerate warped product, and let be the vector fields $X,Y,Z\in\fivectlift{B \times F,B}$ and $U,V,W\in\fivectlift{B \times F,F}$. Let $\kosz$ be the Koszul form on $B \times_f F$, and $\kosz_B, \kosz_F$ the lifts of the Koszul forms on $B$, respectively $F$. Then
\begin{enumerate}
	\item \label{thm_wp_deg_koszul:BBB}
	$\kosz(X,Y,Z)=\kosz_B(X,Y,Z)$.
	\item \label{thm_wp_deg_koszul:BBF}
	$\kosz(X,Y,W) = \kosz(X,W,Y) = \kosz(W,X,Y) = 0$.
	\item \label{thm_wp_deg_koszul:BFF}
	$\kosz(X,V,W) = \kosz(V,X,W) = -\kosz(V,W,X) = f \metric{V,W}_F X(f)$.
	\item \label{thm_wp_deg_koszul:FFF}
	$\kosz(U,V,W)=f^2\kosz_F(U,V,W)$.
\end{enumerate}
\end{proposition}
\begin{proof}
\eqref{thm_wp_deg_koszul:BBB} and \eqref{thm_wp_deg_koszul:FFF} follow from properties of the lifts of vector fields, the Definition \ref{def_Koszul_form} of the Koszul form, and the equation \eqref{eq_wp_metric}.

\eqref{thm_wp_deg_koszul:BBF}
By Definition \ref{def_Koszul_form},
\begin{equation*}
\begin{array}{llll}
	\kosz(X,Y,W) &=&\dsfrac 1 2 \{ X \metric{Y,W} + Y \metric{W,X} - W \metric{X,Y} \\
	&&\ - \metric{X,[Y,W]} + \metric{Y, [W,X]} + \metric{W, [X,Y]}\}
\end{array}
\end{equation*}
We apply the Proposition \ref{thm_wp_deg_fundam}. From the relation \eqref{thm_wp_deg_fundam:mixed_metric},
\begin{equation*}
\metric{Y,W}=\metric{W,X}=\metric{W,[X,Y]}=0,
\end{equation*}
from the relation \eqref{thm_wp_deg_fundam:mixed_lie_bracket} $[Y,W]=[W,X]=0$, from the relation \eqref{thm_wp_deg_fundam:metric_B_constant_F} $W \metric{X,Y}=0$. Therefore $\kosz(X,Y,W)=0$.

From \eqref{thm_Koszul_form_props_commutYZ} of the Theorem \ref{thm_Koszul_form_props} we obtain that
\begin{equation*}
	\kosz(X,W,Y) = X \metric{W,Y} - \kosz(X,Y,W) = 0.
\end{equation*}

From \eqref{thm_Koszul_form_props_commutXY} of the Theorem \ref{thm_Koszul_form_props} and from Proposition \ref{thm_wp_deg_fundam}\eqref{thm_wp_deg_fundam:mixed_lie_bracket} we obtain that
\begin{equation*}
	\kosz(W,X,Y) = \kosz(X,W,Y) - \metric{[X,W],Y}= 0.
\end{equation*}

\begin{equation*}
\begin{array}{lrl}
	\tag{\ref{thm_wp_deg_koszul:BFF}} \kosz(X,V,W) &:=&\dsfrac 1 2 \{ X \metric{V,W} + V \metric{W,X} - W \metric{X,V} \\
	&&\ - \metric{X,[V,W]} + \metric{V, [W,X]} + \metric{W, [X,V]}\}\\
	&=& \dsfrac 1 2 X \metric{V,W}
\end{array}
\end{equation*}
from Proposition \ref{thm_wp_deg_fundam}, using it as in the property \eqref{thm_wp_deg_koszul:BBF} of the present Proposition. By applying the property \eqref{thm_wp_deg_fundam:metric_F_var_B} we have $\kosz(X,V,W) = f \metric{V,W}_F X(f)$. From Theorem \ref{thm_Koszul_form_props} property \eqref{thm_Koszul_form_props_commutXY},
$$\kosz(V,X,W)=\kosz(X,V,W)-\metric{[X,V],W},$$
but since $[X,V]=0$, $\kosz(V,X,W) = f \metric{V,W}_F X(f)$ as well.

From Theorem \ref{thm_Koszul_form_props} property \eqref{thm_Koszul_form_props_commutYZ},
$$\kosz(V,W,X) = V\metric{W,X} - \kosz(V,X,W),$$
but since $\metric{W,X}=0$, the property \eqref{thm_wp_deg_koszul:BFF} of the present Proposition shows that 
$$\kosz(V,W,X) = -f\metric{V,W}_F X(f).$$
\end{proof}

Further, we will study some properties of the warped products, in situations when the warping function $f$ is allowed to vanish or to become negative, and when $(B,g_B)$ and $(F,g_F)$ are allowed to be singular and with variable signature.

%~~~~~~~~~~~~~~~~~~~~~~~~~~~~~~~~~~~~~~~~~~~~~~~~~~~~~~~~~~~~~~~~~~~~~~~%
\section{Degenerate warped products of {\semireg} manifolds}
\label{s_semi_reg_semi_riem_man_warped}

In the following we will provide the condition for a degenerate warped product of {\semireg} {\semiriem} manifolds to be a {\semireg} {\semiriem} manifold.

\begin{theorem}
\label{thm_rad_stat_semi_riem_man_warped}
Let $(B,g_B)$ and $(F,g_F)$ be two {\rstationary} {\semiriem} manifolds, and $f\in\fiscal{B}$ a smooth function so that $\de f\in\annihforms B$. Then, the warped product manifold $B \times_f F$ is a {\rstationary} {\semiriem} manifold.
\end{theorem}
\begin{proof}
We have to show that $\kosz(X,Y,W)=0$ for any $X,Y\in\fivect{B \times_f F}$ and $W\in\fivectnull{B \times_f F}$. It is enough to check this for vector fields which are lifts of vector fields $X_B,Y_B,W_B\in \fivectlift{B \times F,B}$,
$X_F,Y_F,W_F\in \fivectlift{B \times F,F}$, where $W_B,W_F\in\fivectnull{B \times_f F}$. Then, from the Proposition \ref{thm_wp_deg_koszul}:
\begin{enumerate}
	\item 
	$\kosz(X_B,Y_B,W_B)=\kosz_B(X_B,Y_B,W_B)=0$,
	\item
	$\kosz(X_B,Y_B,W_F) = \kosz(X_B,Y_F,W_B) = \kosz(X_F,Y_B,W_B) = 0$,
	\item
	$\kosz(X_B,Y_F,W_F) = \kosz(Y_F,X_B,W_F) = f \metric{Y_F,W_F}_F X_B(f) = 0$, because $\metric{Y_F,W_F}_F=0$, and\\
	$\kosz(X_F,Y_F,W_B) = -f \metric{X_F,Y_F}_F W_B(f) = 0$, from $W_B(f)=0$,
	\item
	$\kosz(X_F,Y_F,W_F)=f^2\kosz_F(X_F,Y_F,W_F) = 0$.
\end{enumerate}
\end{proof}

\begin{theorem}
\label{thm_semi_reg_semi_riem_man_warped}
Let $(B,g_B)$ and $(F,g_F)$ be two {\semireg} {\semiriem} manifolds, and $f\in\fiscal{B}$ a smooth function so that $\de f\in\srformsk 1 B$. Then, the warped product manifold $B \times_f F$ is a {\semireg} {\semiriem} manifold.
\end{theorem}

\begin{proof}
All contractions of the form $\kosz(X,Y,\cocontr)\kosz(Z,T,\cocontr)$ are well defined, according to Theorem \ref{thm_rad_stat_semi_riem_man_warped}. 
From Proposition \ref{thm_sr_cocontr_kosz}, it is enough to show that they are smooth.
It is enough to check this for vector fields which are lifts of vector fields $X_B,Y_B,Z_B,T_B\in \fivectlift{B \times F,B}$,
$X_F,Y_F,Z_F,T_F\in \fivectlift{B \times F,F}$.
Let's denote by $\cocontr_B$ and $\cocontr_F$ the symbol for the covariant contraction on $B$, respectively $F$.
Then, from the Proposition \ref{thm_wp_deg_koszul}:
\begin{equation*}
\begin{array}{lll}
\kosz(X_B,Y_B,\cocontr)\kosz(Z_B,T_B,\cocontr)&=&\kosz(X_B,Y_B,\cocontr_B)\kosz(Z_B,T_B,\cocontr_B) \\
&&+\kosz(X_B,Y_B,\cocontr_F)\kosz(Z_B,T_B,\cocontr_F) \\
&=&\kosz_B(X_B,Y_B,\cocontr_B)\kosz_B(Z_B,T_B,\cocontr_B) \\
&\in& \fiscal{B \times_f F}. \\

\kosz(X_B,Y_B,\cocontr)\kosz(Z_F,T_B,\cocontr)&=&\kosz(X_B,Y_B,\cocontr)\kosz(Z_B,T_F,\cocontr)\\
&=&\kosz(X_B,Y_B,\cocontr_B)\kosz(Z_B,T_F,\cocontr_B) \\
&&+\kosz(X_B,Y_B,\cocontr_F)\kosz(Z_B,T_F,\cocontr_F) = 0. \\

\kosz(X_B,Y_B,\cocontr)\kosz(Z_F,T_F,\cocontr)&=&\kosz(X_B,Y_B,\cocontr_B)\kosz(Z_F,T_F,\cocontr_B) \\
&&+\kosz(X_B,Y_B,\cocontr_F)\kosz(Z_F,T_F,\cocontr_F) \\
&=&-\kosz_B(X_B,Y_B,\cocontr_B)f \metric{Z_F,T_F}_F \de f(\cocontr_B) \\
&=&-f\metric{Z_F,T_F}_F (\nabla^B_{X_B}{Y_B})(\de f) \\
&\in& \fiscal{B \times_f F}. \\

\kosz(X_B,Y_F,\cocontr)\kosz(T_F,Z_B,\cocontr)&=&\kosz(X_B,Y_F,\cocontr)\kosz(Z_B,T_F,\cocontr)\\
&=&\kosz(X_B,Y_F,\cocontr_B)\kosz(Z_B,T_F,\cocontr_B) \\
&&+\kosz(X_B,Y_F,\cocontr_F)\kosz(Z_B,T_F,\cocontr_F) \\
&=&f \metric{Y_F,\cocontr_F}_F X_B(f)\kosz(Z_B,T_F,\cocontr_F) \\
&=& f^3 X_B(f)\kosz_F(Z_B,T_F,Y_F) \\
&\in& \fiscal{B \times_f F}. \\
\end{array}
\end{equation*}

\begin{equation*}
\begin{array}{lll}
\kosz(X_B,Y_F,\cocontr)\kosz(Z_F,T_F,\cocontr)&=&\kosz(X_B,Y_F,\cocontr_B)\kosz(Z_F,T_F,\cocontr_B) \\
&&+\kosz(X_B,Y_F,\cocontr_F)\kosz(Z_F,T_F,\cocontr_F) \\
&=&f^3 X_B(f) \metric{Y_F,\cocontr_F}_F\kosz_F(Z_F,T_F,\cocontr_F) \\
&=&f^3 X_B(f) \kosz_F(Z_F,T_F,Y_F) \\
&\in& \fiscal{B \times_f F}. \\
\end{array}
\end{equation*}
\end{proof}

\begin{remark}
Even though $(B,g_B)$ and $(F,g_F)$ are {\nondeg} {\semiriem} manifolds, if the function $f$ becomes $0$, the warped product manifold $B\times_f F$ is a singular {\semiriem} manifold.
\end{remark}

\begin{corollary}
Let's consider that $(B,g_B)$ is a {\nondeg} {\semiriem} manifold, and let $f\in\fiscal{B}$. If $(F,g_F)$ is {\rstationary}, then the warped product $B\times_f F$ also is {\rstationary}. If $(F,g_F)$ is {\semireg}, then the warped product $B\times_f F$ also is {\semireg}. In particular, if both manifolds $(B,g_B)$ and $(F,g_F)$ are {\nondeg}, and the warping function $f\in\fiscal{B}$, then $B\times_f F$ is {\semireg}.
\end{corollary}
\begin{proof}
If the manifold $(B,g_B)$ is {\nondeg}, then any function $f\in\fiscal{B}$ also satisfies $\de f\in\annihforms B$ and $\de f\in\srformsk 1 B$. Then the corollary follows from Theorems \ref{thm_rad_stat_semi_riem_man_warped} and \ref{thm_semi_reg_semi_riem_man_warped}.
\end{proof}

\begin{proposition}[The case $f\equiv 0$]
$B\times_0 F$ is a singular {\semiriem} manifold with degenerate metric of constant $\rank g=\dim B$.
\end{proposition}
\begin{proof}
The proof can be found in \citep{Kup87b}{287}.
In fact, Kupeli does even more in \cite{Kup87b}, by showing that any {\rstationary} {\semiriem} manifold is locally a warped product of the form $B\times_0 F$.
\end{proof}

\begin{remark}
The warped product of {\nondeg} {\semiriem} manifolds stays {\nondeg} for $f>0$. If $f\to 0$, we can see for example from \cite{ONe83} that the connection $\nabla$ (\citep{ONe83}{206--207}), the Riemann curvature $R_\nabla$ (\citep{ONe83}{209--210}), the Ricci tensor $\ric$ and the scalar curvature $s$ (\citep{ONe83}{211}) diverge in general.
\end{remark}

%~~~~~~~~~~~~~~~~~~~~~~~~~~~~~~~~~~~~~~~~~~~~~~~~~~~~~~~~~~~~~~~~~~~~~~~%
\section{Riemann curvature of {\semireg} warped products}
\label{s_riemann_wp_deg}

In this section we will assume $(B,g_B)$ and $(F,g_F)$ to be {\semireg} {\semiriem} manifolds, $f\in\fiscal{B}$ a smooth function so that $\de f\in\srformsk 1 B$, and $B\times_f F$ the warped product of $B$ and $F$.
The central point is to find the relation between the Riemann curvature $R$ of $B\times_f F$ and those on $(B,g_B)$ and $(F,g_F)$. The relations are similar to those for the {\nondeg} case \citepcf{ONe83}{210--211} for the Riemann curvature operator $R(\_,\_)$, but since this operator is not well defined and is divergent for degenerate metric, we need to use the Riemann curvature tensor $R(\_,\_,\_,\_)$. The proofs given here are based only on formulae which work for the degenerate case as well.

\begin{definition}
\label{def_hessian}
Let $(M,g)$ be a {\semireg} {\semiriem} manifold. The \textit{Hessian} of a scalar field $f$ satisfying $\de f\in\metricformsk 1 M$ is the smooth tensor field $H^f\in\tensors 0 2{M}$ defined by
\begin{equation}
	H^f(X,Y) := \left(\der_X\de f\right)(Y)
\end{equation}
for any $X,Y\in\fivect{M}$.
\end{definition}

\begin{theorem}
\label{thm_wp_nondeg_riemm_tens}
Let $B \times_f F$ be a degenerate warped product of {\semireg} {\semiriem} manifolds with $f\in\fiscal{B}$ a smooth function so that $\de f\in\srformsk 1 B$, and $R_B, R_F$ the lifts of the Riemann curvature tensors on $B$ and $F$. Let $X,Y,Z,T\in\fivectlift{B \times F,B}$, $U,V,W,Q\in\fivectlift{B \times F,F}$, and let $H^f$ be the \textit{Hessian} of $f$ (which exists because $\de f\in\srformsk 1 B$, see Definition \ref{def_hessian}. Then:
\begin{enumerate}
	\item \label{thm_wp_nondeg_riemm_tens:BBBB}
	$R(X,Y,Z,T) = R_B(X,Y,Z,T)$
	\item \label{thm_wp_nondeg_riemm_tens:BBBF}
	$R(X,Y,Z,Q) = 0$
	\item \label{thm_wp_nondeg_riemm_tens:BBFF}
	$R(X,Y,W,Q) = 0$
	\item \label{thm_wp_nondeg_riemm_tens:FFBF}
	$R(U,V,Z,Q) = 0$
	\item \label{thm_wp_nondeg_riemm_tens:BFFB}
	$R(X,V,W,T) = -fH^f(X,T)\metric{V,W}_F$
	\item \label{thm_wp_nondeg_riemm_tens:FFFF}
	$\begin{aligned}[t]
          R(U,V,W,Q)=&R_F(U,V,W,Q) \\
          &+ f^2 \annihprod{\de f,\de f}_B\big(\metric{U,W}_F\metric{V,Q}_F \\
					&- \metric{V,W}_F\metric{U,Q}_F\big)
       \end{aligned}$
\end{enumerate}
the other cases being obtained by the symmetries of the Riemann curvature tensor.
\end{theorem}
\begin{proof}
In order to prove these identities, we will use the Koszul formula for the Riemann curvature from equation \eqref{eq_riemann_curvature_tensor_koszul_formula}. We will denote the covariant contraction with $\cocontr$ on $B \times_f F$, and with $\stackrel B{\cocontr}$ and $\stackrel F{\cocontr}$ on $B$, respectively $F$.

\begin{equation*}
\begin{array}{llll}
	\eqref{thm_wp_nondeg_riemm_tens:BBBB}&
	R(X,Y,Z,T)&=& X \kosz(Y,Z,T) - Y \kosz(X,Z,T) - \kosz([X,Y],Z,T)\\
	&&& + \kosz(X,Z,\cocontr)\kosz(Y,T,\cocontr) - \kosz(Y,Z,\cocontr)\kosz(X,T,\cocontr) \\
	&&=& X \kosz(Y,Z,T) - Y \kosz(X,Z,T) - \kosz([X,Y],Z,T)\\
	&&& + \kosz(X,Z,\stackrel B{\cocontr})\kosz(Y,T,\stackrel B{\cocontr}) - \kosz(Y,Z,\stackrel B{\cocontr})\kosz(X,T,\stackrel B{\cocontr}) \\
	&&=& R_B(X,Y,Z,T),
	\end{array}
\end{equation*}
where we applied \eqref{thm_wp_deg_koszul:BBF} from the Proposition \ref{thm_wp_deg_koszul}.
\begin{equation*}
\begin{array}{llll}
	\eqref{thm_wp_nondeg_riemm_tens:BBBF}&
	R(X,Y,Z,Q)&=& X \kosz(Y,Z,Q) - Y \kosz(X,Z,Q) - \kosz([X,Y],Z,Q)\\
	&&& + \kosz(X,Z,\cocontr)\kosz(Y,Q,\cocontr) - \kosz(Y,Z,\cocontr)\kosz(X,Q,\cocontr) \\
	&&=& \kosz(X,Z,\cocontr)\kosz(Y,Q,\cocontr) - \kosz(Y,Z,\cocontr)\kosz(X,Q,\cocontr) \\
	&&=& \kosz(X,Z,\stackrel B{\cocontr})\kosz(Y,Q,\stackrel B{\cocontr}) - \kosz(Y,Z,\stackrel B{\cocontr})\kosz(X,Q,\stackrel B{\cocontr}) \\
	&&=& 0,
	\end{array}
\end{equation*}
by the same property, which also leads to
\begin{equation*}
\begin{array}{llll}
	\eqref{thm_wp_nondeg_riemm_tens:BBFF}&
	R(X,Y,W,Q)&=& X \kosz(Y,W,Q) - Y \kosz(X,W,Q) - \kosz([X,Y],W,Q)\\
	&&& + \kosz(X,W,\cocontr)\kosz(Y,Q,\cocontr) - \kosz(Y,W,\cocontr)\kosz(X,Q,\cocontr) \\
	&&=& \kosz(X,W,\cocontr)\kosz(Y,Q,\cocontr) - \kosz(Y,W,\cocontr)\kosz(X,Q,\cocontr) \\
	&&=& 0.
	\end{array}
\end{equation*}
\begin{equation*}
\begin{array}{llll}
	\eqref{thm_wp_nondeg_riemm_tens:FFBF}&
	R(U,V,Z,Q)&=& U \kosz(V,Z,Q) - V \kosz(U,Z,Q) - \kosz([U,V],Z,Q)\\
	&&& + \kosz(U,Z,\cocontr)\kosz(V,Q,\cocontr) - \kosz(V,Z,\cocontr)\kosz(U,Q,\cocontr) \\
	&&=& U \left(f\metric{V,Q}_FZ(f)\right) - V \left(f\metric{U,Q}_FZ(f)\right) \\
	&&&- f\metric{[U,V],Q}_FZ(f) \\
	&&& + \kosz(U,Z,\stackrel B{\cocontr})\kosz(V,Q,\stackrel B{\cocontr}) - \kosz(V,Z,\stackrel B{\cocontr})\kosz(U,Q,\stackrel B{\cocontr}) \\
	&&& + \kosz(U,Z,\stackrel F{\cocontr})\kosz(V,Q,\stackrel F{\cocontr}) - \kosz(V,Z,\stackrel F{\cocontr})\kosz(U,Q,\stackrel F{\cocontr}) \\
	&&=& f Z(f) \left(U \metric{V,Q}_F - V \metric{U,Q}_F - \metric{[U,V],Q}_F\right) \\
	&&& + \kosz(U,Z,\stackrel F{\cocontr})\kosz(V,Q,\stackrel F{\cocontr})_F - \kosz(V,Z,\stackrel F{\cocontr})\kosz(U,Q,\stackrel F{\cocontr})_F \\
	&&=& f Z(f) \left(U \metric{V,Q}_F - V \metric{U,Q}_F - \metric{[U,V],Q}_F\right) \\
	&&& + f\metric{U,\stackrel F{\cocontr}}_FZ(f) \kosz(V,Q,\stackrel F{\cocontr})_F \\
	&&& -f\metric{V,\stackrel F{\cocontr}}_FZ(f)\kosz(U,Q,\stackrel F{\cocontr})_F \\
	&&=& f Z(f) (U \metric{V,Q}_F - V \metric{U,Q}_F - \metric{[U,V],Q}_F ) \\
	&&& + \kosz(V,Q,U)_F - \kosz(U,Q,V))_F \\
	&&=& 0, \\
	\end{array}
\end{equation*}
where we used \eqref{thm_wp_deg_koszul:BFF} and  \eqref{thm_wp_deg_koszul:FFF} from the Proposition \ref{thm_wp_deg_koszul}, together with the Definition \ref{def_Koszul_form}. We also used the property that the covariant contraction on $F$ cancels the coefficient $f^2$ of $\kosz(U,V,W)_F$.

\begin{equation*}
\begin{array}{llll}
	\eqref{thm_wp_nondeg_riemm_tens:BFFB}&
	R(X,V,W,T)&=& X \kosz(V,W,T) - V \kosz(X,W,T) - \kosz([X,V],W,T)\\
	&&& + \kosz(X,W,\cocontr)\kosz(V,T,\cocontr) - \kosz(V,W,\cocontr)\kosz(X,T,\cocontr) \\
	&&=& - X \left(f T(f) \metric{V,W}_F\right) \\
	&&& - \kosz(V,W,\stackrel B{\cocontr})\kosz(X,T,\stackrel B{\cocontr}) \\
	&&& + \kosz(X,W,\stackrel F{\cocontr})\kosz(V,T,\stackrel F{\cocontr})_F \\
	&&=& - X \left(f T(f) \metric{V,W}_F\right) \\
	&&& + f\metric{V,W}_F\de f(\cocontr) \kosz(X,T,\stackrel B{\cocontr})_B \\
	&&& + X(f) \metric{W,\stackrel F{\cocontr}}_F T(f) \metric{V,\stackrel F{\cocontr}}_F \\
	&&=& - X(f) T(f) \metric{V,W}_F - fX(T(f)) \metric{V,W}_F \\
	&&& + f\metric{V,W}_F \kosz(X,T,\stackrel B{\cocontr})_B\de f(\stackrel B{\cocontr}) \\
	&&& + X(f) T(f) \metric{W,V}_F \\
	&&=& f\metric{V,W}_F \left[\kosz(X,T,\stackrel B{\cocontr})_B\de f(\stackrel B{\cocontr}) - X(T(f))\right] \\
	&&=& f\metric{V,W}_F \left[\kosz(X,T,\stackrel B{\cocontr})_B\de f(\stackrel B{\cocontr}) - X\metric{T,\grad f}_B\right] \\
	&&=& -f H^f(X,T)\metric{V,W}_F, \\
	\end{array}
\end{equation*}
where we applied the definition of the Hessian for {\semireg} {\semiriem} manifolds, for $f$ so that $\de f\in\srformsk 1 B$, and the properties of the Koszul derivative of warped products, as in the Proposition \ref{thm_wp_deg_koszul}.

\begin{equation*}
\begin{array}{llll}
	\eqref{thm_wp_nondeg_riemm_tens:FFFF}&
	R(U,V,W,Q)&=& U \kosz(V,W,Q) - V \kosz(U,W,Q) - \kosz([U,V],W,Q)\\
	&&& + \kosz(U,W,\cocontr)\kosz(V,Q,\cocontr) - \kosz(V,W,\cocontr)\kosz(U,Q,\cocontr) \\
	&&=& R_F(U,V,W,Q)\\
	&&& + \kosz(U,W,\stackrel B{\cocontr})\kosz(V,Q,\stackrel B{\cocontr}) - \kosz(V,W,\stackrel B{\cocontr})\kosz(U,Q,\stackrel B{\cocontr}) \\		
	&&=& R_F(U,V,W,Q)\\
	&&& + f^2 \metric{U,W}_F\de f(\stackrel B{\cocontr})\metric{V,Q}_F\de f(\stackrel B{\cocontr}) \\
	&&& - f^2 \metric{V,W}_F\de f(\stackrel B{\cocontr})\metric{U,Q}_F\de f(\stackrel B{\cocontr}) \\
\end{array}
\end{equation*}

\begin{equation*}
\begin{array}{llll}
	&&=& R_F(U,V,W,Q)\\
	&&& + f^2 \annihprod{\de f,\de f}_B\big(\metric{U,W}_F\metric{V,Q}_F \\
	&&&- \metric{V,W}_F\metric{U,Q}_F\big). \\
\end{array}
\end{equation*}
\end{proof}

\begin{remark}
Despite the fact that the Riemann tensor $R(\_,\_)$ is divergent when the warping function converges to $0$ even for warped products of {\nondeg} metrics (\citep{ONe83}{209--210}), Theorem \ref{thm_wp_nondeg_riemm_tens} shows again that the Riemann curvature tensor $R(\_,\_,\_,\_)$ is smooth.
\end{remark}

%~~~~~~~~~~~~~~~~~~~~~~~~~~~~~~~~~~~~~~~~~~~~~~~~~~~~~~~~~~~~~~~~~~~~~~~%
\section{Applications to General Relativity}
\label{s_riemann_wp_deg_applications}

In this section we show how the degenerate warped product can be used to construct cosmological models and to model black holes.

The degenerate warped product allows the warping function to become $0$ at some points. Under the hypothesis of the Theorem \ref{thm_riemann_curvature_semi_regular} the Riemann curvature still remains well-defined and smooth. 
As we shown in \cite{Sto11a}, for a smooth Riemann curvature tensor of a four-dimensional {\semireg} manifold we can write a \textit{densitized version of Einstein's equation} which remains smooth, and which reduces to the standard version if the metric is non-degenerate:
\begin{equation}
\label{eq_einstein_idx:densitized}
	G\det g + \Lambda g\det g = \kappa T\det g,
\end{equation}
where $G=\ric- \frac 1 2 sg$, $T$ is the stress-energy tensor, $\kappa:=\dsfrac{8\pi \mc G}{c^4}$, $\mc G$ is Newton's constant and $c$ the speed of light.

The generalization of the warped product we propose here provides a powerful method to resolve singularities in cosmology. If we show that a singularity can be obtained as a {\semireg} warped product of {\semireg} (in particular non-degenerate) manifolds, it follows that the densitized version of the Einstein equation is smooth at that singularity.

%~~~~~~~~~~~~~~~~~~~~~~~~~~~~~~~~~~~~~~~~~~~~~~~~~~~~~~~~~~~~~~~~~~~~~~~%
\subsection{Cosmological models}
\label{s_riemann_wp_deg_applications:FLRW}

In the following we mention a result showing that the \textit{Friedmann-Lema\^itre-Robertson-Walker} spacetime can be extended beyond the Big Bang singularity \cite{Sto11h,Sto12a}.

If $(\Sigma,g_{\Sigma})$ is a connected three-dimensional Riemannian manifold of constant sectional curvature $k\in\{-1,0,1\}$ (\ie $H^3$, $\R^3$ or $S^3$) and $a\in (t_1,t_2)$, $-\infty\leq t_1 < t_2\leq \infty$, $a\geq 0$, then the warped product $(t_1,t_2) \times_a \Sigma$ is called a \textit{Friedmann-Lema\^itre-Robertson-Walker} spacetime:
\begin{equation}
	g = -\de t\otimes\de t + a^2(t) g_{\Sigma}
\end{equation}

By allowing $a$ to become $0$, we can construct Friedmann-Lema\^itre-Robertson-Walker cosmological models in which the evolution equation can pass through the singularities.

In \cite{Sto11h,Sto12a},  we applied the technique presented here in detail and we proved that the Friedmann-Lema\^itre-Robertson-Walker spacetime, and the densitized version of the Einstein equation, can be extended smoothly at and beyond the Big Bang singularity. These techniques also allowed the construction of very general {\semireg} cosmological models which are inhomogeneous and anisotropic, and satisfy the Weyl curvature hypothesis \cite{Sto12c}.

%~~~~~~~~~~~~~~~~~~~~~~~~~~~~~~~~~~~~~~~~~~~~~~~~~~~~~~~~~~~~~~~~~~~~~~~%
\subsection{Stationary black holes}
\label{s_riemann_wp_deg_applications:black_holes}

At the black hole singularities, some components of the metric become infinite, so they are apparently different than the singularities due to the degeneracy of a smooth metric. In the case of the event horizon, Eddington \cite{eddington1924comparison} and Finkelstein \cite{finkelstein1958past} were able to remove the singularities by applying a singular coordinate transformation (which obviously implies a change of the atlas). It is known that this can be done only for the event horizon singularities, while the $r=0$ singularities are genuine. However, in \cite{Sto11e,Sto11f,Sto11g} we successfully applied the method of coordinate changes to make the metric smooth and actually analytic at the $r=0$ singularities. 

We will show here the simplest case, that of the {\schw} black hole \cite{Sto11e}. The metric in {\schw} coordinates is
\begin{equation}
\label{eq_schw_schw}
\de s^2 = -\dsfrac{r-2m}{r}\de t^2 + \dsfrac{r}{r-2m}\de r^2 + r^2\de\sigma^2,
\end{equation}
where $\de\sigma^2 = \de\theta^2 + \sin^2\theta \de \phi^2$.

The coordinate transformation that makes the metric analytic and {\semireg} everywhere, including at $r=0$, is
\begin{equation}
\begin{array}{l}
\bigg\{
\begin{array}{ll}
r &= \tau^2 \\
t &= \xi\tau^4 \\
\end{array}
\\
\end{array}
\end{equation}
The metric becomes in the new coordinates:
\begin{equation}
\label{eq_schw_analytic_tau_xi}
\de s^2 = -\dsfrac{4\tau^4}{2m-\tau^2}\de \tau^2 + (2m-\tau^2)\tau^{4}\(4\xi\de\tau + \tau\de\xi\)^2 + \tau^4\de\sigma^2,
\end{equation}

The metric remains, of course, singular at $r=0$, because the Kretschmann scalar $R_{abcd}R^{abcd}$ is singular at $r=0$. 
But the part of the singularity due to the coordinates is removed, and what remains is an analytic degenerate metric, which is {\semireg} \cite{Sto11e}.

One crucial step in moving from the coordinates $(\xi,\tau)$ to $(\xi,\tau, \phi, \rho)$ was the usage of degenerate warped products, present because of the spherical symmetry of the {\schw} solution. The warped product was also used for the {\rn} black hole solutions \cite{Sto11e,Sto11f}, and the {\kn} solutions (which have cylindrical symmetry) \cite{Sto11g}. These simple stationary solutions are shown to be compatible with global hyperbolicity, and therefore with a unitary evolution \cite{Sto12e}.

As a side effect of the degenerate warped product, it turned out that the resulting singularities are accompanied by dimensional reduction effects like those needed in Quantum Gravity \cite{Sto12d}.

%~~~~~~~~~~~~~~~~~~~~~~~~~~~~~~~~~~~~~~~~~~~~~~~~~~~~~~~~~~~~~~~~~~~~~~~%
\section*{Acknowledgment}

Partially supported by Romanian Government grant PN II Idei 1187.

%~~~~~~~~~~~~~~~~~~~~~~~~~~~~~~~~~~~~~~~~~~~~~~~~~~~~~~~~~~~~~~~~~~~~~~~%

\end{document}